\documentclass[colorinlistoftodos,12pt]{article}

\usepackage[utf8]{inputenc}

\usepackage[utf8]{inputenc}
\usepackage[a4paper]{geometry}
\usepackage{amsmath,amsfonts,amsthm,epsfig,latexsym,graphicx,amssymb,xfrac,faktor}
\usepackage{enumerate} 
\usepackage[english]{babel}

\usepackage{mathtools}
\usepackage{stmaryrd} 

\usepackage[dvipsnames]{xcolor} 
\usepackage{tikz-cd}

\usepackage[
  colorlinks = true,  
  linkcolor = RoyalBlue, 
  citecolor = RoyalBlue 
  ]{hyperref} 

\usepackage{sectsty} 
\sectionfont{\centering\color{RoyalBlue}} 
\subsectionfont{\centering\color{RoyalBlue}} 

\DeclareTextFontCommand{\emph}{\color{RoyalBlue}\em} 

\newcommand{\RR}{\mathbb{R}}
\newcommand{\NN}{\mathbb{N}}

\newcommand{\ZZ}{\mathbb{Z}}

\newcommand{\TT}{\mathbb{T}}

\newcommand{\Class}{\mathcal{C}}

\newcommand{\algcap}{\mathrlap{\hspace{2.3pt}\cdot}{\cap}}

\makeatletter
  \def\@tvsp{\mathchoice{{}\mkern-4.5mu}{{}\mkern-4.5mu}{{}\mkern-2.5mu}{}}
  \def\ltrivert{\left|\@tvsp\left|\@tvsp\left|}
  \def\rtrivert{\right|\@tvsp\right|\@tvsp\right|}

  \def\ldrivert{\left|\@tvsp\left|}
  \def\rdrivert{\right|\@tvsp\right|}
\makeatother

\makeatletter
\newcommand{\et}[1][\epsilon]{\def\@psodefault{#1}{\@et}}
\newcommand{\@et}[1][T]{$(\@psodefault,#1)$} 

\newcommand{\psos}[1][\epsilon]{
\def\@psodefault{#1}
{\@pso}s
}
\newcommand{\pso}[1][\epsilon]{
\def\@psodefault{#1}
\@pso
}
\newcommand{\@pso}[1][T]{$(\@psodefault,#1)$-pseudo-orbit} 
\makeatother

\newcommand{\Rec}{\mathcal{R}}   

\newcommand{\bfrac}[2]{{
  \raisebox{0.1em}{\scalebox{0.85}{$#1$}}/
  \raisebox{-0.1em}{\scalebox{0.85}{$#2$}}
}}

\DeclareMathOperator{\tor}{tors}

\DeclareMathOperator{\im}{im}
\DeclareMathOperator{\len}{len}

\DeclareMathOperator{\cone}{cone}

\DeclareMathOperator{\conv}{Conv}

\newtheoremstyle{colorplain}%
{\topsep}   
{\topsep}   
{\itshape}  
{0pt}       
{} 
{.}         
{5pt plus 1pt minus 1pt} 
{\textbf{\textcolor{RoyalBlue}{\textbf{\thmname{#1} \thmnumber{#2}}}}\thmnote{ (#3)}}
{}

\newtheoremstyle{colorremark}%
{\topsep}   
{\topsep}   
{}  
{0pt}       
{\itshape} 
{.}         
{5pt plus 1pt minus 1pt} 
{\textcolor{RoyalBlue}{\thmname{#1} \thmnumber{#2}}\thmnote{ (#3)}}
{}  

\newtheoremstyle{colorremarkbis}%
{\topsep}   
{\topsep}   
{}  
{0pt}       
{\itshape} 
{.}         
{5pt plus 1pt minus 1pt} 
{\textcolor{RoyalBlue}{\thmname{#1}\thmnumber{#2}}\thmnote{(#3)}}
{}

\newtheoremstyle{colordefinition}%
{\topsep}   
{\topsep}   
{}  
{0pt}       
{} 
{.}         
{5pt plus 1pt minus 1pt} 
{\textcolor{RoyalBlue}{\textbf{\thmname{#1} \thmnumber{#2}}}\thmnote{ (#3)}}
{}

\theoremstyle{colorplain}
\newtheorem{theorem}{Theorem}

\numberwithin{theorem}{section}

\newtheorem{maintheorem}{Theorem}

\newtheorem{maincorollary}[maintheorem]{Corollary}

\newtheorem{lemma}[theorem]{Lemma}

\theoremstyle{colorremarkbis}

\newtheorem*{remark*}{Remark}

\theoremstyle{colorremark}

\newtheorem{mainremark}[maintheorem]{Remark}

\theoremstyle{colordefinition}

\makeatletter 
\newenvironment{proofabstract}[1][\proofname]{
    \par
    \pushQED{\qed}%
    \normalfont \topsep6\p@\@plus6\p@\relax
    \trivlist
    \item\relax
    {\itshape
    #1\@addpunct{.}}\hspace\labelsep\ignorespaces
}{%
    \popQED\endtrivlist\@endpefalse
}

\renewenvironment{proof}[1][Proof]{
    \setcounter{claim}{0}    
    \setcounter{claimproof}{0}    
    \par
    \pushQED{\qed}%
    \normalfont\topsep6\p@\@plus6\p@\relax
    \trivlist
    \item\relax
    {\itshape\color{RoyalPurple}#1\@addpunct{.}}\hspace\labelsep\ignorespaces
}{%
    \popQED\endtrivlist\@endpefalse
}

\makeatother

\newcounter{claimproof} 

  {\begin{proofabstract}[Sketch proof]}
  {\end{proofabstract}}

\title{Partial section III: for Anosov flows}
\author{Théo Marty}
\date{}

\begin{document}

\maketitle

\begin{abstract}
	In the previous papers in the series, we characterized partial cross-sections for general flows, in the spirit of Fried's work on global cross-sections. In this paper, we deduce several consequences for Anosov flows. 
	
	We provide a homology criterion for the existence of a partial cross-section in a given cohomology class.
	Additionally, there are at most finitely many partial cross-sections in that cohomology class. 
	We deduce that on a 3-dimensional hyperbolic manifold, any Anosov flow is homologically full.
\end{abstract}

In dimension 3, transverse tori play a central role in the understanding of Anosov flows. For non-transitive systems, Smale~\cite{smale1967differentiable} partitioned the ambient manifold of a hyperbolic flow into finitely many minimal pieces, separated by transverse surfaces. Brunella~\cite{brunella1993separating} reinforced the decomposition for non-transitive Anosov flows, by partitioning the manifold into minimal pieces separated by incompressible transverse tori. Béguin-Bonatti-Yu~\cite{beguin2016spectral} gave a finer partition of the ambient manifold, by adding non-null-cohomologous transverse tori. Then, conversely, Béguin-Bonatti-Yu~\cite{beguin2017building} gave a quite general construction of Anosov flows by gluing minimal pieces along their transverse boundaries.

For transitive Anosov flows, Anosov suspensions are the typical examples of transitive Anosov flows with a transverse torus.
Bonatti-Langevin~\cite{bonatti1994exemple} built an example of a transitive Anosov flow that admits a transverse torus but that is not a suspension. The construction was later generalized by Barbot~\cite{barbot1998generalizations}. 
In both cases, Fenley~\cite{fenley1994anosov} characterized the stable/unstable foliations on transverse tori.

In parallel, Fried~\cite{Fried82} described the set of global cross-sections of a general flow. Very few Anosov flows admit a global cross-section, so Fried's arguments are of little help in providing meaningful information for a general Anosov flow. In a recent work~\cite{marty2025partialI,marty2025partialII}, we gave a description of the set of \emph{partial cross-sections} of a general flow (a transverse closed surface). These sections are much more abundant on Anosov flows. And in particular, they correspond to the transverse tori added by Béguin-Bonatti-Yu. 

In the present paper, we give three consequences for Anosov flows. In the process, we give arguments to detect the tori described in~\cite{beguin2016spectral}. The argument also holds in higher dimensions, where very little is known about transverse cross-sections.

\paragraph{Results.} We consider an Anosov flow $\phi$ on a closed manifold $M$, first in any dimension. We later restrict our attention to dimension 3. Fried \cite{Fried82} introduced a set $D_\phi$ of asymptotic directions. Note that for Anosov flows, the convex hull $\conv(D_\phi)$ is spanned by the set of $\frac{1}{\len(\gamma)}[\gamma]$ where $\gamma$ runs over the periodic orbits of $\phi$.

\begin{maintheorem}\label{thm-cross-section}
	Let $\phi$ be an Anosov flow (in any dimension) and $\alpha$ in $H^1(M,\ZZ)$ be non-zero. There exists a partial cross-section cohomologous to $\alpha$ if and only if $\alpha(D_\phi)\geq 0$ holds. 
	
	Additionally, if $\phi$ is transitive, then there are at most finitely many partial cross-sections cohomologous to $\alpha$, up to isotopy along the flow.
\end{maintheorem}

For completeness, we mention the similar result for null-cohomologous partial cross-sections, which is well-known.

\begin{mainremark}
	Let $\phi$ be an Anosov flow. It admits a null-cohomologous partial cross-section if and only if $\phi$ is not transitive (see \cite{smale1967differentiable}). Additionally, when it is not transitive, $\phi$ admits a countable (infinite) number of null-cohomologous partial cross-sections that are not isotopic along the flow (see \cite[Theorem~4.3]{marty2025partialII}).
\end{mainremark}

Note that Mosher~\cite{Mosher1989,Mosher1990} proved the first claim in Theorem~\ref{thm-cross-section} for some pseudo-Anosov flows, which include all transitive Anosov flows. Theorem~\ref{thm-ps-existence} can be replaced by Mosher's results in some parts of this paper. We should also mention~\cite{Landry2024}, which provides stronger conclusions for pseudo-Anosov flows with no perfect fits.

The rest of our results consider the homology classes of periodic orbits more carefully. The flow is said to be \emph{homologically full} if every homology class in $H_1(M,\ZZ)$ is represented by a periodic orbit. 

\begin{maintheorem}\label{thm-hom-full-section}
	Let $\phi$ be a transitive Anosov flow (in any dimension). Then $\phi$ is homologically full if and only if it admits no partial cross-section. 
\end{maintheorem} 

Regarding known Anosov flows, Anosov suspensions are not homologically full, but skewed $\RR$-covered Anosov flows are known to be homologically full (see~\cite[Lemma 5]{marty2025skewed}). As discussed further down, there exist non-$\RR$-covered Anosov flows that are homologically full, but the Bonatti-Langevin example is not homologically full.

We record the following special case for homology spheres, and a criterion in the spirit of Sullivan \cite{Sullivan1976}.

\begin{maincorollary}
	On a homology sphere, any transitive Anosov flow has no partial cross-section.
\end{maincorollary}

\begin{maincorollary}
	Let $\phi$ be a transitive Anosov flow (in any dimension). It admits a partial cross-section if and only if $\conv(D_\phi)$ does not contain zero in its interior.
\end{maincorollary}

In the non-transitive case, the condition in Theorem \ref{thm-hom-full-section} does not suffice (see Theorem~\ref{thm-example-non-t}). We do not know which condition should replace it.

\vspace{\baselineskip}

In dimension 3, Theorem~\ref{thm-hom-full-section} yields information about Anosov flows on hyperbolic manifolds.

\begin{maintheorem}\label{thm-hyp-implies-hom-full}
	On a 3-dimensional hyperbolic manifold, any Anosov flow is homologically full.
\end{maintheorem}

Note that Anosov flows on hyperbolic manifolds are known to exhibit geometric behaviors. Those that are $\RR$-covered are topologically equivalent to Reeb-Anosov flows~\cite{marty2023skewed}. Those that are not $\RR$-covered were proven to be quasi-geodesic by Fenley~\cite{fenley2026non}.

\vspace{\baselineskip}

Our characterization also provides an alternative and elementary proof of the following theorem of Sharp. 

\begin{maintheorem}[Sharp, Theorem 1 in~\cite{sharp}]\label{thm-Sharp}
	Let $\phi$ be a transitive Anosov flow. Then $\phi$ is homologically full if and only if there exists a $\phi$-invariant probability measure that is totally supported and null-homologous.
\end{maintheorem}

We give the alternative proof in the appendix. Note that Sharp's original proof relies on the thermodynamic formalism, which is quite different from ours.

\paragraph{Acknowledgment.} The present research was conducted at the Max-Planck-Institut for Mathematics in Bonn, and at the Institut de Mathématique de Bourgogne in Dijon (with the bourse Brunella). I thank both institutes for their hospitality.


\section{Preliminaries}\label{sec-prelem}

A \emph{partial cross-section} is an embedded compact surface, of class $\Class^1$, and transverse to the flow. A partial cross-section $S$ is naturally co-oriented by the flow, yielding a cohomology class denoted by $[S]$ in $H^1(M,\ZZ)$. A \emph{global cross-section} is a partial cross-section that intersects every orbit of the flow. Global cross-sections are well-understood. It is useful to consider partial cross-sections up to \emph{isotopy along the flow}, that is, up to isotopy among partial cross-sections. 

Take some small $\epsilon>0$. 
We equip $M$ with an arbitrary Riemannian metric.
A periodic \emph{$\epsilon$-pseudo-orbit} is a periodic curve $\gamma\colon\ZZ/L\ZZ\to M$ such that there exists a finite set $\Delta\subset\ZZ/L\ZZ$ with:
\begin{itemize}
	\item any two points in $\Delta$ are at distance at least $1$,
	\item $\gamma$ is continuous outside $\Delta$,
	\item on each complementary component of $\Delta$, $\gamma$ coincides with an orbit arc (that is, $\phi_s(\gamma(t))=\gamma(t+s)$ holds for any $t\not\in\Delta$ and any small enough $s>0$),
	\item for any $s\in\Delta$, the left and right limits of $\gamma$ at $s$ lie at distance at most $\epsilon$ (in $M$) from each other.
\end{itemize}

Below, we denote by $\len(\gamma)=L$ the length of $\gamma$.
The following fact is well-known and very useful to simplify the study of partial cross-sections.

\begin{lemma}[Shadowing Lemma~\cite{katok1995introduction}]
    Let $\phi$ be an Anosov flow. For any $\eta>0$ there exists $\epsilon>0$ so that any (periodic or not) $\epsilon$-pseudo-orbit $\gamma$ of $\phi$ is $\eta$-shadowed by a unique (real) orbit $\delta$ of $\phi$. That is, there is a reparametrization $\sigma\colon\RR\to\RR$ of $\delta$ so that $d(\delta\circ\sigma(t),\gamma(t))<\eta$ holds for all $t\in\RR$.
\end{lemma}

Note that when $\epsilon$ is smaller than half the injectivity radius of $M$, one can associate a canonical homology class $[\gamma]$ in $H_1(M,\ZZ)$ to an $\epsilon$-pseudo-orbit $\gamma$, by connecting the jumps in $\gamma$ by $\epsilon$-small curves. When $\epsilon$ is small enough, $[\gamma]$ is equal to the homology class of the unique periodic orbit that shadows $\gamma$. Denote by $D_{\phi,\epsilon}$ the compact convex set defined by

$$D_{\phi,\epsilon}=\overline\conv\left\{\frac{1}{\len(\gamma)}[\gamma],\gamma\text{ a periodic $\epsilon$-pseudo-orbit}\right\}.$$

Then define $D_\phi$ as the decreasing intersection of the $D_{\phi,\epsilon}$.

\begin{theorem}[Fried~\cite{Fried82}]
	Let $\phi$ be a (not necessarily Anosov) $\Class^1$ flow and take $\alpha$ in $H^1(M,\ZZ)$. The class $\alpha$ is cohomologous to a global cross-section if and only if $\alpha(D_\phi)>0$ holds. Additionally, any two global cross-sections that are cohomologous are isotopic along the flow.
\end{theorem}

Note that Fried's theorem uses a different definition for $D_\phi$, but their convex hulls are equal (see~\cite[Appendix]{marty2025partialI}), so the above version holds.

\begin{theorem}[Marty~\cite{marty2025partialII}]\label{thm-ps-existence}
	Let $\phi$ be a continuous flow and $\alpha$ in $H^1(M,\ZZ)$ be non-zero. There exists a partial cross-section cohomologous to $\alpha$ if and only if $\alpha(D_{\phi,\epsilon})\geq0$ holds for some small $\epsilon>0$.
\end{theorem}

The same article also provides a description of the set of all partial cross-sections up to isotopy along the flow, see~\cite[Theorem B]{marty2025partialII}. We will not need the full description here, but we will use one of its consequences: a finiteness criterion. We refer to~\cite{marty2025partialII} for the construction on general flows. Here, we only consider Anosov flows, for which the definitions can be quite simplified. We let the reader verify that, using the shadowing lemma, the definitions given below and in the reference are equivalent for Anosov flows.

We fix a class $\alpha$ in $H^1(M,\ZZ)$, non-zero and with $\alpha(D_\phi)\geq0$. Call \textbf{$\alpha$-recurrence set}, denoted by $\Rec_\alpha$, the closure of the union of the periodic orbits $\gamma$ that satisfy $\alpha(\gamma)=0$. Its connected components are called the \emph{$\alpha$-recurrence chains}. We define the oriented graph $G_{\phi,\alpha}$ whose vertices are the $\alpha$-recurrence chains (possibly empty), and, given two $\alpha$-recurrence chains $R_1,R_2$, there is an edge $R_1\to R_2$ if the unstable lamination of $R_1$ intersects the stable lamination of $R_2$.

\begin{theorem}[Marty~\cite{marty2025partialII}]\label{thm-ps-finiteness}
	Assume that $\phi$ is Anosov and take $\alpha$ non-zero in $H^1(M,\ZZ)$ so that there exists a partial cross-section cohomologous to $\alpha$. The set of partial cross-sections cohomologous to $\alpha$ is finite if and only if $G_{\phi,\alpha}$ is finite and transitive (as an oriented graph).
\end{theorem}

\section{Proof of Theorem~\ref{thm-cross-section}}\label{sec-proof-A}

The proof of Theorem~\ref{thm-cross-section} is a direct consequence of the classification of partial cross-sections in the general case and of the shadowing lemma.

\begin{proof}[Proof of the first part in Theorem~\ref{thm-cross-section}]
	When $\phi$ is Anosov and $\epsilon>0$ is small enough, any periodic $\epsilon$-pseudo-orbit is shadowed by a real periodic orbit, and in particular, it is homologous to that periodic orbit. It follows that $\cone(D_{\phi,\epsilon})=\cone(D_{\phi})$ when $\epsilon$ is small enough. The conclusion follows from Theorem~\ref{thm-ps-existence}.
\end{proof}

When the flow is Anosov and transitive, from standard arguments, the hypothesis in Theorem~\ref{thm-ps-finiteness} is verified. 

\begin{proof}[Proof of the second part in Theorem~\ref{thm-cross-section}]
	We now assume that $\phi$ is transitive. 
	Take two periodic orbits $\gamma_1,\gamma_2$ with $\alpha(\gamma_i)=0$, and that pass close enough to each other. The stable leaf of $\gamma_1$ intersects the unstable leaf of $\gamma_2$, so $\gamma_1$ and $\gamma_2$ belong to the same $\alpha$-recurrence chain. This implies that each $\alpha$-recurrence chain is open in the $\alpha$-recurrence set. So by compactness, $G_{\phi,\alpha}$ is finite.

	Since $\phi$ is transitive, there exists a periodic orbit $\delta$ that passes very close to every $\alpha$-recurrence chain.
	Given two $\alpha$-recurrence chains $R_1,R_2$, one can truncate $\delta$ into an $\epsilon$-pseudo-orbit that spends an infinite amount of time in $R_1$ (in the past), then jumps to $\delta$, follows $\delta$, jumps to $R_2$ and remains inside $R_2$ for an infinite amount of time. Using the shadowing lemma, there is a unique orbit $\gamma$ of $\phi$ that shadows that pseudo-orbit. Then $\gamma$ converges toward $R_1$ in the past, so it belongs to its unstable lamination. Similarly, it belongs to the stable lamination of $R_2$. It follows that $G_{\phi,\alpha}$ is transitive. Then according to Theorem~\ref{thm-ps-finiteness}, there are at most finitely many partial cross-sections up to isotopy.
\end{proof}

\section{Homologically full Anosov flows}\label{sec-hull-homol}

We prove the following.
Denote by $X_\phi$ the subset of $H_1(M,\ZZ)$ consisting of the homology classes of the periodic orbits of $\phi$. The flow $\phi$ is said to be \emph{homologically full} if $X_\phi$ is equal to $H_1(M,\ZZ)$. We say that $\phi$ is \emph{homologically full modulo torsion} if the image of $X_\phi$ in $\bfrac{H_1(M,\ZZ)}{\tor}$ is the whole module.

From~\cite[Lemma 2.3]{marty2025partialI} and from the periodic shadowing lemma, there exist $\epsilon>0$ and finitely many periodic orbits $\delta_1,\ldots,\delta_n$, so that any periodic $\epsilon$-pseudo-orbit of $\phi$ is homologous to a sum (with non-negative weights) of the $[\delta_i]$. It follows:

\begin{lemma}
	For any Anosov flow $\phi$ (not necessarily transitive), $\cone(D_\phi)$ is a polyhedral cone spanned by finitely many points that have integer coefficients. 
\end{lemma}

It is also known to be a consequence of the existence of a Markov partition for transitive Anosov flows.

\begin{lemma}
	A transitive Anosov flow is homologically full if and only if it is homologically full modulo torsion, if and only if $\cone(D_\phi)=H_1(M,\RR)$ holds, if and only if $\conv(D_\phi)$ contains zero in its interior.
\end{lemma}

The direct implications clearly hold, and the last two statements are clearly equivalent.

\begin{proof}
	Assume that $\cone(D_\phi)=H_1(M,\RR)$ holds.
	Then there exist finitely many periodic orbits $\gamma_1,\ldots,\gamma_n$ such that the convex hull of the $[\gamma_i]$ contains zero in its interior.
	From~\cite{adachi1987closed}, $H_1(M,\ZZ)$ is spanned by the homology classes of $\phi$. So up to adding some $\gamma_i$, we may assume that $\sum_i\ZZ[\gamma_i]=H_1(M,\ZZ)$ holds.
	
	From above, there exist some scalars $\lambda_i>0$ that satisfy $\sum_i\lambda_i[\gamma_i]=0$. Since the $[\gamma_i]$ have integer coefficients, the $\lambda_i$ can be taken with rational coefficients, and even with integer coefficients.
	Using the above relation, we get that $-\lambda_i[\gamma_i]$ belongs to $\sum_i\NN[\gamma_i]$, which implies $\sum_i\NN[\gamma_i]=\sum_i\ZZ[\gamma_i]$.

	Take a periodic orbit $\delta$ of $\phi$ that passes close enough to each $\gamma_i$. For any $c$ in $H_1(M,\ZZ)$, $c-[\delta]$ can be written as a sum $\sum_ia_i[\gamma_i]$ for some $a_i\geq 0$. It implies $c=[\delta]+\sum_ia_i[\gamma_i]$. One can construct a pseudo-orbit that follows $\delta$ once and that follows each $\gamma_i$ $a_i$ times. Then there exists a periodic orbit $\gamma$ that shadows it. It follows that $c$ is homologous to $\gamma$. Hence $\phi$ is homologically full.
\end{proof}

\begin{proof}[Proof of Theorem~\ref{thm-hom-full-section}]
	Note that when $\phi$ is transitive, any partial cross-section is necessarily non-null-cohomologous.

	Assume first that $\phi$ admits a non-null-cohomologous partial cross-section $S$. Then for any periodic orbit $\gamma$ of $\phi$, we have $S\algcap\gamma\geq 0$. From~\cite[Lemma 3.6]{marty2025partialI} and from the periodic shadowing lemma, we have $S\algcap D_\phi\geq0$. So $D_\phi$ lies in a half space of $H_1(M,\RR)$. In particular, $\phi$ is not homologically full.
	
	Assume now that $\phi$ is not homologically full. 
	It follows from above that $\conv(D_\phi)$ is contained in a half space (including the boundary of the half space). So there exists $\alpha$ in $H^1(M,\RR)$, non-zero, that satisfies $\alpha(D_\phi)\geq0$. From above, $\cone(D_\phi)$ is described by equations with integer coefficients, so $\alpha$ may be taken in $H^1(M,\ZZ)$. 
	From~\cite[Theorem A]{marty2025partialII},~\cite[Lemma 3.6]{marty2025partialI} and the shadowing lemma, there exists a partial cross-section of $\phi$ cohomologous to $\alpha$.
\end{proof}

\begin{proof}[Proof of Theorem~\ref{thm-hyp-implies-hom-full}]
	Note that in dimension 3, any partial cross-section is necessarily a torus, since it is foliated by the transverse intersection of the stable/unstable foliations. From~\cite{brunella1993separating}, transverse tori of Anosov flows are incompressible, and from~\cite[Theorem 9]{preissmann1942quelques}, hyperbolic manifolds have no incompressible tori. Thus an Anosov flow on a hyperbolic manifold is transitive (non-transitive Anosov flows admit transverse tori~\cite{brunella1993separating}) and has no partial cross-section. It follows from Theorem~\ref{thm-hom-full-section} that it is homologically full.
\end{proof}

\section{An example of non-transitive Anosov flows}\label{sec-example}

We give an example of a non-transitive Anosov flow which is homologically full.
Let $\Sigma$ be a hyperbolic surface, $M=T^1\Sigma$ be its unit tangent bundle, $\phi$ be the geodesic flow on $M$, and $\gamma$ be a periodic orbit of $\phi$. The flow $\phi$ is well-known to be homologically full (see~\cite[Lemma 5]{marty2025skewed}).

\begin{lemma}
	The map $H_1(M\setminus\gamma,\ZZ)\to H_1(M,\ZZ)$ is an isomorphism if and only if there exists a closed surface $S$ transverse to $\gamma$, and with $|S\cap\gamma|=1$.
\end{lemma}

\begin{proof}
	The long exact sequence in homology contains
	$$\cdots\to H_2(M,M\setminus\gamma,\ZZ)\xrightarrow{f} H_1(M\setminus\gamma,\ZZ)\xrightarrow{g} H_1(M,\ZZ)\xrightarrow{h} H_1(M,M\setminus\gamma,\ZZ)=0.$$
	It follows from $\im(g)=\ker(h)=H_1(M,\ZZ)$ that $g$ is surjective. We have $\ker(g)=\im(f)$, so $g$ is injective if and only if any meridian of $\gamma$ is null-homologous in $H_1(M\setminus\gamma,\ZZ)$, that is, if and only if there exists a closed surface $S$ transverse to $\gamma$, and with $|S\cap\gamma|=1$.
\end{proof}

Let $\gamma$ be the geodesic induced by a simple oriented closed geodesic on $\Sigma$. Given any geodesic $\delta$ that intersects (transversally) $\gamma$ only once, the torus $\TT_\delta$, consisting of unit vectors based at $\delta$, intersects $\gamma$ exactly once. So from above, $H_1(M\setminus\gamma,\ZZ)=H_1(M,\ZZ)$. It follows that the restriction of $\phi$ to $M\setminus\gamma$ is homologically full. 

\begin{theorem}\label{thm-example-non-t}
	Let $\phi_1,\phi_2$ be the derived Anosov flows of $\phi$ on $\gamma$, so that $\gamma$ is attracting for $\phi_1$ and repelling for $\phi_2$. Drill a transverse neighborhood of $\gamma$ in $\phi_1$ and~$\phi_2$. Then any Anosov gluing of these two flows is non-transitive and homologically full.
\end{theorem}

The theorem easily follows from~\cite{beguin2017building}.

\appendix

\section{Alternative proof of Theorem~\ref{thm-Sharp}}

To an invariant probability measure $\mu$ of $\phi$ corresponds a unique homology class $[\mu]_\phi$ such that for any closed 1-form $\beta$ on $M$, we have $\beta([\mu]_\phi)=\int_M\beta(\frac{\partial\phi_t}{\partial t})d\mu$. We say that $\mu$ is \emph{null-homologous} if $[\mu]_\phi=0$.

\begin{lemma}
	Let $\phi$ be a transitive Anosov flow and $\mu$ be a $\phi$-invariant probability measure. If $\mu$ is totally supported, then $[\mu]_\phi$ belongs to the interior of $\cone(D_\phi)$.
\end{lemma}

\begin{proof}
	From~\cite[Appendix]{marty2025partialI}, $[\mu]_\phi$ lies inside $\cone(D_\phi)$. To prove that it lies in the interior of the cone, we reason by contradiction. So assume that it lies on the boundary of $\cone(D_\phi)$. Then there exists a non-zero class $\alpha$ in $H^1(M,\RR)$ that satisfies $\alpha(D_\phi)\geq 0$ and $\alpha([\mu]_\phi)=0$. Since $\cone(D_\phi)$ is polyhedral and described by equations with integer coefficients, we may choose $\alpha$ in $H^1(M,\ZZ)$. 
	
	From Theorem~\ref{thm-cross-section}, $\phi$ admits a partial cross-section $S$ cohomologous to $\alpha$. Take $\epsilon>0$ small so that the $2\epsilon$-flow box around $S$ is embedded. Denote by $h\colon\RR\to\RR$ a smooth function that is non-decreasing, with $h'>0$ on $]-\epsilon,\epsilon[$, and that satisfies $h(]-\infty,-\epsilon])=0$ and $h([\epsilon,+\infty[)=1$. Also define the function $f\colon\phi_{[-\epsilon,\epsilon]}(S)\to[0,1]$ by $f\circ\phi_t(x)=h(t)$ for any $x\in S$ and $t\in[-\epsilon,\epsilon]$. Note that $df$ is a closed 1-form on $\phi_{[-\epsilon,\epsilon]}(S)$, and that it satisfies $df\equiv0$ on $\partial\phi_{[-\epsilon,\epsilon]}(S)$.

	We can extend $df$ into a closed 1-form $\beta$ on $M$, by $\beta\equiv0$ outside $\phi_{[-\epsilon,\epsilon]}(S)$. By construction, $\beta$ is closed, it satisfies $\beta(\frac{\partial\phi_t}{\partial t})\geq 0$, and the inequality is strict on a neighborhood of $S$. It follows 
	$$\beta([\mu]_\phi)=\int_M\beta\left(\frac{\partial\phi_t}{\partial t}\right)d\mu\geq\int_{\phi_{]-\epsilon,\epsilon[}(S)}\beta\left(\frac{\partial\phi_t}{\partial t}\right)d\mu >0,$$
	which contradicts the assumption.
\end{proof}

\begin{proof}[Proof of Theorem~\ref{thm-Sharp}]
	Assume that $\phi$ is homologically full. Take $\mu$ to be any totally supported invariant probability measure. Then one can find periodic orbits $\gamma_i$ and $a_i>0$ with $\sum_ia_i=1$ so that $\sum_ia_i[\gamma_i]=-[\mu]_\phi$ holds. Then the renormalization of the sum of $\mu$ and of the Lebesgue measures on the $\gamma_i$, weighted according to the $a_i$, is null-homologous. So there exists a null-homologous totally supported $\phi$-invariant probability measure.

	Assume now that $\phi$ is not homologically full. Then from Theorem~\ref{thm-hom-full-section}, $\cone(D_\phi)$ does not contain zero in its interior. So from above, no totally supported invariant probability measure is null-homologous.
\end{proof}

Note that $\phi$ may have a null-homologous invariant probability measure (or a null-homologous periodic orbit), but its support must be disjoint from any partial cross-section.

\bibliographystyle{alpha}
\bibliography{ref}

\end{document}